\documentclass[11pt]{amsart}

\usepackage{marginnote}
\usepackage{amsmath}
\usepackage{amssymb}
\usepackage{amsfonts}
\usepackage{amsmath}
\usepackage{amsthm}
\usepackage{mathtools}
\usepackage{comment}
\usepackage{tikz}
\usepackage{pgfplots}
\pgfplotsset{compat=1.10}
\usetikzlibrary{calc}

\theoremstyle{plain}
\newtheorem{thm}{Theorem}[section]
\newtheorem{prop}[thm]{Proposition}
\newtheorem{lem}[thm]{Lemma}

\newtheorem{cor}[thm]{Corollary}

\theoremstyle{definition}
\newtheorem{defi}[thm]{Definition}
\newtheorem{ex}[thm]{Example}
\newtheorem{rem}[thm]{Remark}

\numberwithin{equation}{section}

\def\C{\mathbb C}
\def\R{\mathbb R}

\def\M{\mathcal M}

\def\Z{\mathbb Z}
\def\Q{\mathbb Q}
\def\ga{\Gamma}
\def\m{\mathbf m}
\def\t{\mathbf t}
\def\z{\mathbf z}
\def\x{\mathbf x}
\def\y{\mathbf y}
\def\p{\mathbf p}

\def\lora{\longrightarrow}

\def\norm{|\!| \cdot |\!|}

\def\mnote{\marginnote}
\DeclareMathOperator{\codim}{codim}
\DeclareMathOperator{\orb}{orb}

\begin{document}

\title[Toric varieties and polyhedral horofunction compactifications]{Toric varieties vs. horofunction compactifications of polyhedral norms}

\author{Lizhen Ji}
\thanks{The first author acknowledges support from NSF  grants DMS 1107452, 1107263, 1107367 GEometric structures And Representation varieties (the GEAR Network)
and partial support from Simons Fellowship (grant \#305526) and the Simons grant  \#353785. }
\author{Anna-Sofie Schilling}
\thanks{The second author was partially supported by the European Research Council under ERC-Consolidator grant 614733, and by the German Research Foundation in the RTG 2229 Asymptotic Invariants and Limits of Groups and Spaces.}

\date{\today}

%\date{\today}

\begin{abstract}
      We establish a natural and geometric 1-1 correspondence between projective toric varieties of dimension $n$ and horofunction compactifications of $\R^n$ with respect to rational polyhedral norms. For this purpose, we explain a topological model of toric varieties. Consequently, toric varieties in algebraic geometry, normed spaces in convex analysis, and horofunction compactifications in metric geometry are directly and explicitly related.
\end{abstract}

\maketitle

\tableofcontents

\section{Introduction}

In this paper we give a correspondence between the three seemingly different concepts \emph{toric varieties}, \emph{horofunction compactifications} and \emph{polyhedral norms}.  Toric varieties provide a basic class of algebraic varieties which are relatively simple. The nonnegative part and the moment map of toric varieties are essential ingredients of the rich structure of toric varieties. The horofunction compactification of metric spaces is a general method to construct compactifications of metric spaces introduced by Gromov \cite[\S 1.2]{gr} in 1981 (see \S \ref{horo-section} below). Finally, polyhedral norms on $\R^n$ give a special class of  normed linear spaces (or Minkowski spaces \cite{th}) and metric spaces (see \S \ref{poly} below).

In this paper we establish an explicit  geometric connection between projective toric
varieties of dimension $n$ and horofunction compactifications of
 $\R^n$ with respect to rational polyhedral norms. 
\begin{thm}\label{main} 
      In every dimension $n\geq 1$, there exists a bijective correspondence between projective toric  varieties $X$ of dimension $n$ and rational polyhedral norms $\norm$  on $\R^n$ up to scaling such that:
      \begin{enumerate}
	  \item The nonnegative part $X_{\geq 0}$ of a projective toric variety $X$ is homeomorphic to the horofunction compactification $\overline{\R^n}^{hor}$  of $\R^n$ with respect to the distance induced by the corresponding polyhedral norm $\norm$.
	  \item Equivalently, the image of the moment map  of the toric variety $X$ is homeomorphic to the horofunction compactification $\overline{\R^n}^{hor}$ of $\R^n$ with respect to the distance induced by the corresponding polyhedral norm $\norm$.
      \end{enumerate}
\end{thm}

This correspondence  is canonical and  given as follows: The unit ball of a rational polyhedral norm $\norm$ is a rational convex polytope $P$ in $\R^n$ which contains
the origin as an interior point, which in turn gives a fan $\Sigma=\Sigma_P$  in $\R^n$ by taking cones over the faces of $P$, and hence gives a toric variety $X=X_\Sigma$. Note that the fan $\Sigma_P$ does not change when the polytope $P$ is scaled, and hence the correspondence is up to scaling on the polyhedral norms $|\!| \cdot |\!|$. 

This result adds another perspective on the close relations between integral convex polytopes and toric projective varieties, for a detailed description see \cite[Chap 2]{od1}. 

Theorem \ref{main} implies the following

\begin{cor}
      Let $|\!| \cdot |\!|$ be a rational polyhedral norm on $\R^n$, and $P$ its unit ball. Let  $P^\circ$ be the polar set\footnote{For a definition see Equation \ref{polar} on page \pageref{polar}.} of $P$, a polytope dual to $P$. Then the horofunction compactification  of $\R^n$ with respect to $|\!| \cdot |\!|$ is homeomorphic  $P^\circ$.
\end{cor}

This gives a bounded realization of the horofunction compactification of $\R^n$.
The same result holds for the horofunction compactification of any polyhederal norm
on $\R^n$ whether it is rational or not
and is proven in \cite[Theorem 1.2]{js}.

It is well-known that algebro-geometric and cohomology properties of toric varieties $X_\Sigma$ are determined by combinatorial and convex properties 
of their fans $\Sigma$ (see \cite{fu} \cite{cls} \cite{od1}).
Consequently, the existence of a correspondence between toric varieties and polyhedral norms is not surprising.
But it is probably not obvious that there exists such a direct
connection between  {\em  horofunction compactifications of $\R^n$  in metric geometry} and   {\em 
 important parts of toric varieties} $X_\Sigma$: the nonnegative part $X_{\Sigma, \geq 0}$ and the image of the moment map of $X_\Sigma$.

The correspondence in Theorem \ref{main} can give numerical invariants of toric varieties. 
In the statement of Theorem \ref{main}, we have fixed the standard  integral structure
$\Z^n$  of $\R^n$ when we discuss toric varieties and the {\em  rationality} of polyhedral norms.
Consequently, by requiring the standard basis of $\Z^n$ to be unit vectors,
we can also fix the standard Euclidean metric on $\R^n$. Though scaling of the polyhedral 
norm $|\!| \cdot |\!|$ does change its unit ball, i.e., the polytope,  it does not change 
the fan $\Sigma_P$ induced from $P$.
On the other hand, we can use the following {\em canonical
normalization} of polyhedral norms: Every vertex of the unit ball of  $|\!| \cdot |\!|$  is integral, and one of them
is primitive. 

With this normalization,  by Theorem \ref{main},
 each projective toric variety $X_\Sigma$ gives 
a unique polyhedral norm $|\!| \cdot |\!|$ and its unit ball $P$, which is a polytope in $\R^n$.
Besides computing the volume of $P$ with respect to the standard Euclidean metric on $\R^n$, 
we can also compute the volume of $P$ with respect to
a suitable notion of  volume induced from the norm  $|\!| \cdot |\!|$.
According to \cite{at},
there are four commonly used definitions of volumes on normed 
vector spaces: Busemann volume, Holmes-Thompson volume,
Gromov volume, and Benson volume.
Consequently, we obtain the following corollary:
\begin{cor}
      Choose one of the four volumes mentioned above, for each projective toric variety $X$, there is a canonical number given by the volume of the unit ball $P$ of the normalized polyhedral norm corresponding to the toric variety $X$.
\end{cor}
One natural question is the meaning of such volumes for toric varieties.
Note that if we use the standard Euclidean metric on $\R^n$,
then the volume of the convex polytope $P$  is related to the implicit degree 
of  the projective toric variety $X_{\Sigma_P}$. See \cite[\S 5]{so}.

 The correspondence in Theorem \ref{main} also raises the question of how to understand toric varieties by using metric geometry.

\vspace{.1in}
\noindent{\bf Horofunction compactification and noncommutative geometry}
\vspace{.1in}

Before we explain some detailed definitions of toric varieties and
horofunction compactifications in later sections,
we point out a connection between horofunction compactifications of normed vector spaces
and reduced $C^\star$-algebras of discrete groups and consequently the noncommutative geometry.

After the horofunction compactification of a proper metric space was
introduced by Gromov \cite[\S 1.2]{gr} in 1981,
the horofunction compactification of a complete simply
connected nonpositively curved manifold was identified
with the geodesic compactification in \cite[\S 3]{bgs}. 
This gives a direct connection between the geometry of geodesics and
analysis, or rather a class of special functions, on the manifold.  
Among nonpositively curved simply connected
Riemannian manifolds, horofunctions are difficult to compute except
for symmetric spaces of noncompact type (see \cite{ha} \cite{gjt}).  
For noncompact locally symmetric spaces of nonpositive curvature,
the horofunction compactification was identified in \cite{jm} and \cite{dfs}.
It will be seen below that 
$\R^n$ with polyhedral norms provide another class of spaces for which all horofunctions can be computed \cite{wa1}. 

It turned out   that the  horofunction
compactifications of $\R^n$ with respect to norms are unexpectedly related to
the noncommutative geometry developed by Alain Connes (see \cite{co}, \cite{ri1} and \cite{ri2}). This brought
another perspective to horofunction compactifications
and motivated the work in this paper.

Let $\ga$  be a countable discrete group such as $\Z^n$ and $SL(n, \Z)$.
Let $C_c(\ga)$ be the convolution $\star$-algebra of complex valued functions of finite support,
i.e., of compact support, 
 on $\ga$. Let $\pi$ be the usual $\star$-representation of $C_c(\ga)$ on $\ell^2(\ga)$.
Then the norm-completion of $\pi(C_c(\ga))$ in the space of operators of $\ell^2(\ga)$ is the {\em reduced $C^\star$-algebra}
 of $\ga$, denoted by $C_r^\star(\ga)$.

Let $\ell$ be a {\em length function} of $\ga$, i.e., a function  $\ell: \ga\to \R^+$ such that \linebreak 
(1) $\ell(e)=0$, 
(2) for all $g\in \ga$, $\ell(g^{-1})=\ell(g)$,
(3) for all $g_1, g_2\in \ga$, $\ell (g_1 g_2)\leq \ell (g_1) \ell(g_2)$.

For example, the word length on $\ga$ with respect to a set of generators gives rise to such a length function.

Let $M_\ell$ be the multiplication operator on $\ell^2(\ga)$ 
defined by the length function $\ell$, which is usually unbounded.
Then $\M_\ell$ serves as a Dirac operator in the noncommutative 
geometry of $C^\star_r(\ga)$.
The following fact is true: {\em For every $f\in C_c(\ga)$,
the commutator $[M_\ell, \pi(f)]$ is a bounded operator on
$\ell^2(\ga)$.}
This allows one to define a {\em semi-norm} on $C_c(\ga)$: 
\begin{equation*}\label{semi-norm}
      L_\ell(f)=|\!| [M_\ell, \pi(f)] |\!|.
\end{equation*}

In general, if $L$ is a semi-norm on a dense sub-$\star$-algebra $A$
of a unital $C^*$-algebra $\overline{A}$ such that $L(1)=0$,
then Connes \cite{co} (see \cite[p. 606]{ri1})  defined a metric $\rho_L$ on the state
space $S(\overline{A})$ of $\overline{A}$ as follows:
For any two states $\mu, \nu\in S(\overline{A})$,
\[
      \rho_L(\mu, \nu) \coloneqq \sup\{|\mu(a) -\nu(a)| \mid a\in A, L(a)\leq 1\}.
\]

We recall that a state on a $C^\star$-algebra $\overline{A}$  is a positive linear functional of norm 1. The set of all states of a $C^\star$-algebra $\overline{A}$
is denoted by $S(\overline{A})$, and is a convex subset of the space of linear functionals of $\overline{A}$.
Extreme points of $S(\overline{A})$ are called pure states of $\overline{A}$.
When $\overline{A}=C^0(X)$, the space of continuous functions
on a compact topological space $X$, then states on $\overline{A}$ correspond
to probability measures on $X$, and pure states correspond to
evaluations on $X$.

In \cite[p. 606]{ri1}, Rieffel called a semi-norm $L$ on $A$ a {\em Lip-norm} if the topology
on $S(\overline{A})$ induced from $\rho_{L_\ell}$ coincides with the
weak $\star$-topology,
and he called a unital $C^\star$-algebra $\overline{A}$
equipped with a Lip-norm 
a {\em compact quantum metric space}.

In \cite{ri1}, Rieffel asked the question: {\em Given a discrete
group $\ga$, is the seminorm $L_\ell$ on $C_c(\ga)$ coming
from a length function on $\ga$ a Lip-norm?}

He could only handle the case $\ga=\Z^n$ and prove the following result:

\begin{prop}[\cite{ri1}, Thm 0.1]
      Let $\ell$ be a length function on $\Z^n$ which is either the word length for some finite generating set or the restriction to $\Z^n$ of some norm on $\R^n$. Then the induced seminorm $L_\ell$ is a Lip-norm on $C_c(\Z^n)$, and hence $C_r^\star(\Z^n)$ is a compact quantum metric space.
\end{prop}

In proving this result, Rieffel made crucial use of horofunction
compactifications of $\R^n$ with respect to norms.
In this paper, he also raised the following question (\cite[Question 6.5]{ri1}): {\em 
Is it true that, for every finite-dimensional vector space and every norm on it, every horofunction (i.e., a boundary point of the horofunction compactification of $\R^n$) is a Busemann function, i.e., the limit of an almost-geodesic ray? }

This question motivated the paper \cite{kmn}  and  was settled completely in \cite{wa1}. 
It also motivated the other papers \cite{wa2}, \cite{wa3},  \cite{wa5}, \cite{agw},  \cite{ww1}, \cite{ww2}, \cite{an},  \cite{de} and \cite{ls} on horofunction compactifications.

\section{Toric varieties}

 In this section, we give a  summary of several results on toric varieties which are needed to understand and prove Theorem \ref{main}. The basic references for this section  are \cite{fu},  \cite{cls},  \cite{od1}, \cite{od2},  \cite{am}, \cite{cox}, and \cite{so}.

 \begin{defi}
 A toric variety over $\C$  is an irreducible variety $V$ over $\C$ such that
 \begin{enumerate}
      \item the complex torus $(\C^*)^n$ is a Zariski dense subvariety of $V$ and
      \item the action of $(\C^*)^n$ on itself by multiplication extends to an action of $(\C^*)^n$ to $V$.
 \end{enumerate}
 \end{defi}

 We  fix the  {\em  standard lattice} $\Z^n$ in $\R^n$, which gives $\R^n$ an {\em integral structure}, and also a $\Q$-structure, $\Q^n\subset \R^n$.

 Recall that a {\em rational polyhedral cone} $\sigma\subset \R^n$ is a cone generated by finitely many elements  $u_1, \cdots, u_m$ of $\Z^n$, or equivalently of $\Q^n$:
 \[
      \sigma=\{\lambda_1 u_1 +\cdots +\lambda_m u_m \in \R^n \mid \lambda_1, \cdots, \lambda_m \geq 0\}.
 \]

 Usually, $\sigma$ is assumed to be {\em strongly convex}: $\sigma \cap -\sigma =\{0\}$, i.e., $\sigma$ does not contain any line through the origin. A {\em face} of a cone $\sigma$ is the intersection of $\sigma$ with the 0-level set of a linear functional which is nonnegative on $\sigma$. The {\em relative interior} and {\em relative boundary} of a cone $\sigma$ are the interior respectively boundary of $\sigma$ in the linear subspace spannend by $\sigma$. 

 For each strongly convex rational polyhedral cone $\sigma$,  define its {\em dual cone}   $\sigma^\vee$ by
 \begin{equation}\label{affine-toric}
      \sigma^\vee \coloneqq \{v\in \R^n \mid \langle v, u\rangle\geq 0, \text{ for all } u\in \sigma\}.
 \end{equation}
 Then $\sigma^\vee$ is also a convex rational polyhedral cone, though it is not strongly convex anymore unless   $\dim \sigma=n$.

 \begin{defi}
      A { fan $\Sigma$  in $\R^n$ is a collection of strongly convex rational polyhedral cones} such that
      \begin{enumerate}
	  \item if $\sigma \in \Sigma$, then every face of $\sigma$ also belongs to $\Sigma$;
	  \item if $\sigma, \tau\in \Sigma$, then their intersection $\sigma\cap \tau$ is a common face of both of them, and hence belongs to $\Sigma$.
      \end{enumerate}
 \end{defi}

 In this paper, we only deal with  fans which consist of finitely many polyhedral cones.
 
 It is known that there is a strong correspondence between fans $\Sigma \subseteq \R^n$ and toric varieties, namely:
 \begin{enumerate} 
      \item For every fan $\Sigma$ of $\R^n$, there is an associated toric variety $X_\Sigma$,
      which is a normal algebraic variety.
      \item If a toric variety $X$ is a normal variety, then $V$ is of the form $X_\Sigma$ for some fan $\Sigma$ in $\R^n$.
 \end{enumerate}

%\mnote{Anna: Do we also require normality in this paper?}
 Because of this correspondence, toric varieties are often required to be normal, for example in \cite{fu}.  In this paper, we follow this convention and require toric varieties to be normal.
 
 The construction of a toric variety $X_\Sigma$ from a fan $\Sigma$ and a description of its topology in terms of $\Sigma$ is crucial to  the proof of Theorem \ref{main}. Therefore we give a short description here:

 Given a fan $\Sigma$ in $\R^n$, its associated {\em  toric variety}  $X_\Sigma$
 is constructed as follows:
 \begin{enumerate}
      \item Each cone $\sigma\in \Sigma$ gives rise to an affine toric variety $U_\sigma$. Specifically, $\sigma^\vee\cap \Z^n$ is  a finitely generated semigroup. Let $\m_1, \cdots, \m_k \in \sigma^\vee\cap \Z^n$ be  a set of generators of this semigroup, i.e., every element of $\sigma^\vee\cap \Z^n$ is of the form $a_1 \m_1+\cdots + a_k \m_k$, with $a_i$ being non-negative integers.  Then the Zariski closure of the image of $(\C^*)^n$ in $\C^k$  under the  embedding  
      \[
	  \varphi: (\C^*)^n \to \C^k, \quad \t \mapsto (\t^{\m_1}, \cdots,  \t^{\m_k})
      \]
      is the affine toric variety $U_\sigma$. Note that we use Laurent monomials for the notation: $\t^{\m_j} = \prod_{l =1}^{n} t_l^{m_{j,l}}$ for all $\t = (t_1, \ldots, t_n) \in (\C^*)^n$ where $m_{j, l}$ denotes the $l$-th component of $\m_j$.  
      \item For any two cones $\sigma_1, \sigma_2$ in $\Sigma$,  if $\sigma_1$ is a face of $\sigma_2$, then $U_{\sigma_1}$  is a Zariski dense subvariety of $U_{\sigma_2}$.
      \item The toric variety $X_\Sigma$ is obtained by gluing  these affine toric varieties $U_\sigma$ together: 
      \[
	  X_\Sigma=\cup_{\sigma\in \Sigma}  U_\sigma/\sim,
      \]
      where the relation $\sim$ is given  by the inclusion relation in (2):  Note that for any two cones $\tau, \sigma\in \Sigma$, the intersection $\tau\cap \sigma$, if nonempty,  is a common face of both $\tau$ and $\sigma$, and hence  $U_{\tau\cap \sigma}$ can be identified with a subvariety of both $U_\tau$ and $U_\sigma$.
 \end{enumerate}
 
 Many properties of $X_\Sigma$ can be expressed in terms of the  combinatorial properties of the fan $\Sigma$.  We state one about the orbits of $(\C^*)^n$ in $X_\Sigma$, details can be found for example in \cite[p. 119]{cls}, \cite[\S 9]{cox} or \cite[\S 3.1]{fu}.
 
 \begin{prop}\label{orbits}
      For every toric variety $X_\Sigma$, there is a bijective correspondence between orbits of the torus $(\C^*)^n$ in $X_\Sigma$ and cones $\sigma$ in the fan $\Sigma$. 
      Denote the orbit in $X_\Sigma$ corresponding to $\sigma$ by $\orb(\sigma) \subset X_\Sigma$. Then $\orb(\sigma)$ is a complex torus isomorphic to $(\C^*)^{n-\dim \sigma}$. In particular, the open and dense orbit $(\C^*)^n$ corresponds to the trivial cone $\{0 \}$. 
 \end{prop}
 
% \vspace{.2in}
%\noindent{\bf A topological model of toric varieties}
% \vspace{.1in}
 
 \subsection{A topological model of toric varieties}
 
 In order to better understand the toric variety $X_\Sigma$  as a compactification of $(\C^*)^n$, we want to give a  topological description of $X_\Sigma$ which  exhibits its
 dependence on $\Sigma$ clearly and also describes explicitly  sequences in $(\C^*)^n$ which converge to points in the complement $X_\Sigma-(\C^*)^n$.
 
 To do so, note that in terms of the standard integral structure $i\Z\subset \C$, we can realize $\C^*$ by:
 \[
      i\Z\backslash \C \cong \C^*, \quad z\mapsto e^{-2\pi z},
 \]
 and when Re$(z)\to +\infty$, it holds $e^{-2\pi z}\to 0$.  Then the exponential map $e^{-2\pi \z}=(e^{-2\pi z_1}, \cdots, e^{-2\pi z_n})$ gives an identification
 \begin{equation*}\label{identification1}
      i\Z^n \backslash \C^n \cong (\C^*)^n.
 \end{equation*}
 Conversely, using the logarithmic function $-\frac{1}{2\pi}\ln$, we get an identification
 \begin{equation}\label{identification2}
      (\C^*)^n\cong i \Z^n \backslash \C^n.
 \end{equation}
 In the following, we denote the complex torus $(\C^*)^n$ by $T$.

 Given any fan $\Sigma$ in $\R^n$, we will define a {\em bordification} $\overline{T}_\Sigma$  of $T=(\C^*)^n$  and show in Proposition \ref{identification3} below that  $\overline{T}_\Sigma$ is homeomorphic to the toric variety  $X_\Sigma$ as $T$-topological spaces. 
 
 \begin{defi}
      For each cone $\sigma\in \Sigma$, define a {\em boundary component}
      \[
	  O(\sigma)=i \Z^n \backslash \C^n/\mathrm{Span}_\C(\sigma).
      \]
 \end{defi}

 Note that this is  a complex torus $(\C^*)^{n-\dim \sigma}$ of dimension equal to $n-\dim \sigma$.  When $\sigma=\{0\}$, then $O(\sigma)=T$. Later we will identify $O(\sigma)$ with $\orb(\sigma)$. 

 \begin{defi}
      Define a {\em topological bordification}  $\overline{T}_\Sigma$ by 
      \begin{equation}\label{orbit}
	  \overline{T}_\Sigma=T \cup \coprod_{\substack{\sigma\in \Sigma, \\ \sigma\neq \{0\}} } O(\sigma)
      \end{equation}
      with the following topology:
      A sequence $\z_j=\x_j+i \y_j \in T=i\Z^n\backslash \C^n$,  where $\x_j\in \R^n$ and $\y_j\in \Z^n\backslash \R^n$, converges to a point $\z_\infty \in O(\sigma)$ for some $\sigma \in \Sigma$ if and only if the following  conditions hold:
      \begin{enumerate}
	  \item The real part $\x_j$ can be written as  $\x_j=\x_j'+\x_j''$ such that when $j\to +\infty$ it holds
	  \begin{enumerate}
	        \item the first part $\x_j'$  is contained in the relative interior of the cone $\sigma$ and its distance to the relative boundary of $\sigma$ goes to infinity,
	        \item the second part $\x_j''$ is bounded. 
	  \end{enumerate}
	  %\mnote{Anna: \red{If the cone $\sigma$ is not contained in $\R^n$, how can the real part $\x_j$ of the sequence be conatined in the relative interior of $\sigma$? See for example the fan obtained by taking cones over the faces of a square and dim($\sigma) = 1$.}} 
	  \item the image of $\z_j$ in $O(\sigma)=i\Z^n\backslash \C^n/\mathrm{Span}_\C(\sigma)$ under the projection
	  \[
	        i\Z^n\backslash \C^n\to i\Z^n\backslash \C^n/\mathrm{Span}_\C(\sigma)
	  \]
	  converges to the point $\z_\infty$.
      \end{enumerate}
 \end{defi}

 Note that the imaginary part $\y_j$ of $\z_j$ lies in the compact torus  $\Z^n\backslash \R^n=(S^1)^n$, and the second condition controls both the imaginary part $\y_j$ and the bounded component
 $\x_j''$ of the real part $\x_j$.
 
 The behaviour of converging sequences is shematically shown in Figure \ref{fig:collapsing_local} and Figure \ref{fig:collapsing_global}.
 
  %\vspace*{3.5in} {\bf Add a picture}
  %\input{pictures_collapsing_fibers_2in1_a}  
  %\input{pictures_collapsing_fibers_2in1_b}

  \begin{figure}[h!]
  \centering
  \includegraphics[scale=0.7]{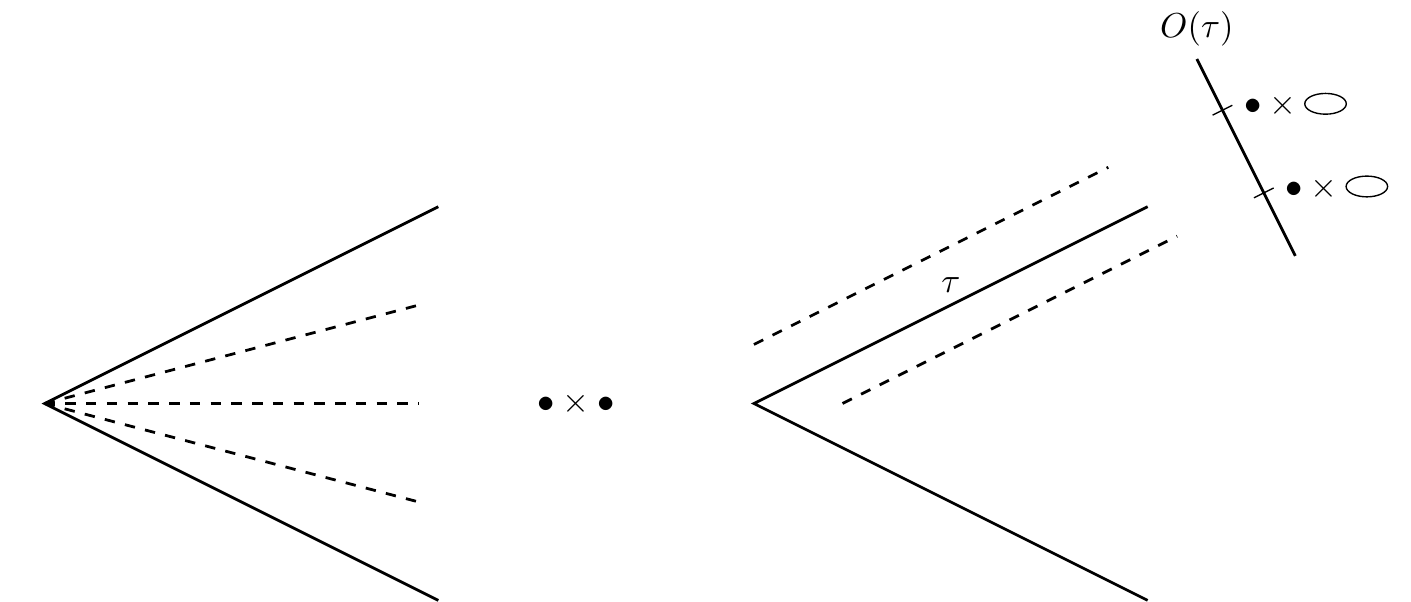}
  \caption{\footnotesize Left: Within a chamber all fibers collapse in the same way: Both circles are collapsed to points. Right: Fibers parallel to a wall collapse differently, depending on the wall and the distance to it. Only one circle is collapsed to a point.} \label{fig:collapsing_local}
  \end{figure}
  
  \begin{figure}[h!]
  \centering
  \includegraphics[scale=0.8]{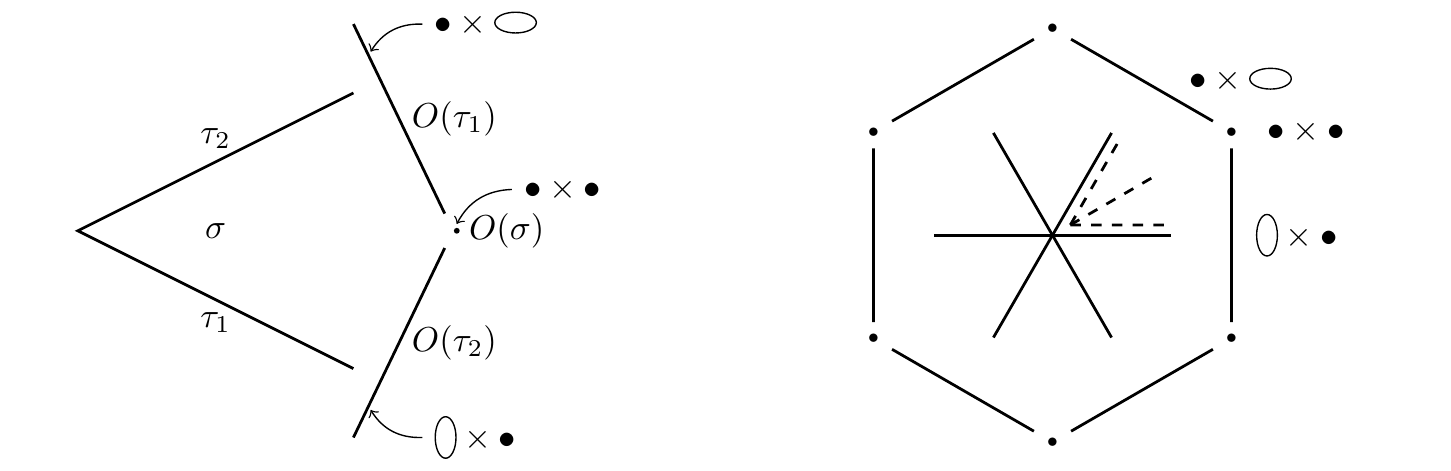}
  \caption{\footnotesize Left: Collapsing behaviour of the fibers when the base point moves to infinity. Depending on the direction of movement either one or both circles are collapsed. Right: Global picture of collapsing of a whole toric variety.} \label{fig:collapsing_global}
  \end{figure}

 \begin{rem}
      The above definition of $\overline{T}_\Sigma$ and the identification of $X_\Sigma$ with $\overline{T}_\Sigma$ in Proposition \ref{identification3}  follows the construction and discussion in \cite[pp. 1-6]{am}. We note that there is one difference with the convention there: On page 2 in \cite{am}, the complex torus $(\C^*)^n$ is identified with $\Z^n \backslash \C^n$, and the real part is the compact torus $(S^1)^n$, and the imaginary part is  $i\R^n$, which can be  identified with $\R^n$. 
 %\mnote{Too many references, no content. State the results and the construction, do not only refer to it.}  
 \end{rem}

 \begin{rem}
      The toric variety  $X_\Sigma$ is compact if and only if  the support of $\Sigma$ is equal to $\R^n$, i.e., $\Sigma$ gives a rational polyhedral decomposition of $\R^n$.
      Similarly, it  is clear from the definition that the bordification  $\overline{T}_\Sigma$  is a compactification of $T$  if and only if the support of $\Sigma$ is equal to $\R^n$.
 \end{rem}

  To obtain a continuous action of $T$ on $\overline{T}_\Sigma$, we note that $\C^n$
  or  $i\Z^n\backslash \C^n$ acts on $T$ and every boundary component $O(\sigma)$
  by translation. These translations are compatible in the following sense.
  
  \begin{lem}\label{translation}
  For any sequence $\z_j \in T=i\Z^n\backslash \C^n$, if $\z_j$ is convergent in $\overline{T}_\Sigma$,
  then for any vector $\z\in \C^n$, or rather its image in $i\Z^n\backslash \C^n$,
  the shifted sequence $\z_j+\z$ is also convergent. Furthermore, 
  \[
      \lim_{n\to+\infty} \z+\z_j =  \z+\lim_{n\to+\infty} \z_j.
  \]
  \end{lem}

  %\mnote{Anna: \red{Reformulate this sentece/add a lemma about parallel shif of sequences.Conditions for convergence do not chcange.}} 
 
 This implies the following result.
 
 \begin{prop}\label{extends}
      The action of $T=(\C^*)^n$ on itself by multiplication extends to a continuous action on $\overline{T}_\Sigma$, and the decomposition in Equation \ref{orbit} into $O(\sigma)$ gives the orbit decomposition of $\overline{T}_\Sigma$ with respect to the action of $T$.
 \end{prop}
 
 \begin{proof}
      We note that the multiplication of the torus $T$ on itself  and the boundary  components $O(\sigma)$ corresponds  to translation in $\C^n$ and $i \Z^n \backslash \C^n$.  Then the proposition follows from Lemma \ref{translation}.
 \end{proof}

 %\mnote{Anna: We do not prove Prop. \ref{extends}., should we give a reference?} 

 One key result we need for the proof of Theorem \ref{main}  is the following description of the toric variety $X_\Sigma$ as a {\em topological $T$-space}. Since this proposition and its proof are not explicitly written down in literature,  we give an outline of the proof on page \pageref{proof_prop_identification} for the convenience of the reader. %, especially the reader working on the  metric geometry and normed vector spaces who might not be familiar with toric varieties.
 
 %\mnote{Anna: write something about the pictures. They are not mentioned in the text and it is not clear what they show and why they are there.}
 
 \begin{prop}\label{identification3}
      The identity map on $T=(\C^*)^n$ extends to a homeomorphism \mbox{$X_\Sigma \to \overline{T}_\Sigma$,} which is  equivariant with respect to the action of $T=(\C^*)^n$, and the $T$-orbits $\mathrm{orb}(\sigma)$ in the toric variety $X_\Sigma$ are mapped homeomorphically to the boundary components $O(\sigma)$.
      %\mnote{Anna: I would add a second part to the proposition about the identification of $\orb(\sigma)$ and $O(\sigma)$, which is also shown in the proof.}
 \end{prop}
 
 The identification between $X_\Sigma$ and $\overline{T}_\Sigma$  allows one  to see that  when a sequence $(\x_j)_j$ of  points in the real part $\R^n$ of the complex  torus $i\Z^n \backslash \C^n$ goes to infinity along  the directions contained  in a cone $\sigma$ of the fan  $\Sigma$,  the sequence $\x_j$  will converge to a point of a  complex torus $i\Z^n\backslash \C^n  /\mathrm{Span}_\C(\sigma)$ of smaller dimension. Hence  the compact  torus $(S^1)^n$, which is the fiber over $x_j$ in the toric variety,  will collapse to a torus of smaller  real dimension $\dim \sigma$.
 
 \begin{rem}
      This result is well-known and can be found for example in \cite[pp. 1-6]{am} \cite[\S 10]{od2}  \cite[p. 211]{cox} or \cite[p. 54]{fu}. %\mnote{check references}
      %The books \cite[pp. 1-6]{am} \cite[\S 10]{od2} study the noncompact part of the torus $\R^n \subset (\C^*)^n$. 
      Such a picture of toric varieties including also the compact part of the torus $(\C^*)^n$  is often described in connection with the moment map of toric varieties (for a reference see \cite[p. 79]{fu} or \cite{mi}). We will come back to this map later. But first we need a description of it 
      as a bordification of $T$ instead of a bounded realization via the image of the moment map.
      A bordification of  the noncompact part  of the torus $\R^n \subset (\C^*)^n$  is described in this way in \cite[p. 6]{am} (see also \cite[\S 10]{od2}) together with a map from the toric variety to this bordification. 
 \end{rem}

 First, we recall some properties of orbits of $T$ in a toric variety $X_\Sigma$.
 The 1-1 correspondence between $T$-orbits in  $X_\Sigma$ and cones in $\Sigma$ mentioned in Proposition \ref{orbits} above  can be described more explicitly (see \cite[p. 28]{fu}, \cite[p. 118]{cls}  and \cite[p. 212]{cox}):

 \begin{prop}\label{distinguished_point}
      For every cone $\sigma\in \Sigma$, there is a {\em distinguished point} $x_\sigma$ in the affine toric variety $U_\sigma\subset X_\Sigma$. It is contained in the orbit $\orb(\sigma)$ of $(\C^*)^n$ in $X_\Sigma$ corresponding to $\sigma$ (see Proposition \ref{orbits}), and hence the orbit $\orb(\sigma)$ is equal to the orbit $(\C^*)^n\cdot x_\sigma$.
 \end{prop}
 
 The distinguished point $x_\sigma$ can be described as follows. 
 The smallest cone  $\{0\}$ of the fan $\Sigma$ corresponds to the affine toric variety $(\C^*)^n$, and the distinguished point is $(1, \cdots, 1)$ in this case. The $T$-orbit through this point gives $T$. 
 
 In general, we note that every one-parameter subgroup $\lambda: \C^* \to (\C^*)^n$ is of the form
 \[
      \lambda_\m(z)=(z^{m_1}, \cdots, z^{m_n}),
 \]
 where $\m=(m_1, \cdots, m_n)\in \Z^n$. Let $\m$ be an integral vector contained in the relative interior of the cone $\sigma$. By \cite[p. 37]{fu} \cite[Proposition 3.2.2]{cls} (see also \cite[p. 212]{cox}), the distinguished point is given by
 \[
      x_\sigma = \lim_{z\to 0} \lambda_\m(z) \in X_\Sigma.
 \]
 
 We need to identify this distinguished point $x_\sigma \in U_\sigma \subset X_\Sigma$ with a corresponding distinguished
 point in the bordification $\overline{T}_\Sigma$. 
 
 \begin{lem}\label{origin} 
      Under the identification of $X_\Sigma$ with $\overline{T}_\Sigma$ in Proposition \ref{identification3},  this  distinguished point $x_\sigma$  in  $X_\Sigma$ corresponds to the image $0_\sigma$ of the origin of $\C^n$ in $\overline{T}_\Sigma$ under the projection $\C^n \to O(\sigma)=i\Z^n \backslash \C^n /\mathrm{Span}_\C(\sigma)$.   When the orbit $O(\sigma)$ is identified with $(\C^*)^r$, where $r = \codim(\sigma) = \dim_\C (\C^n/\mathrm{Span}_\C(\sigma))$, then $0_\sigma$ corresponds to $(1, \cdots, 1)$.
 \end{lem}
 
 \begin{proof}
      As mentioned before, for the trivial cone  $\sigma=\{0\}$ of the fan $\Sigma$, the distinguished point is $(1, \cdots, 1)$.   Under the identification $(\C^*)^n= i\Z^n \backslash \C^n$ in Equation \ref{identification2} on page \pageref{identification2}, the distinguished point $x_\sigma$  corresponds to the image of the origin of $\C^n$ under the projection $\C^n \to i\Z^n \backslash \C^n$.
      
      Therefore, $x_\sigma$ corresponds to $0_\sigma$ in $O(\sigma) \subset \overline{T}_\Sigma$. 
      For any nontrivial cone $\sigma \subset \Sigma$,  the distinguished point $x_\sigma$  in the orbit $\orb(\sigma)$ is equal to the limit $\lim_{z\to 0} \lambda_\m(z)$ in $X_\Sigma$, where $\m$ is an integral vector contained in the relative interior of the cone $\sigma$.
      
      We need to determine the limit  $\lim_{z\to 0}\lambda_\m(z)$ in the bordification $\overline{T}_\Sigma$.  When we identify $(\C^*)^n$ with $\R^n\times i \Z^n \backslash \R^n=i\Z^n \backslash \C^n$ as above in Equation \ref{identification2}, the complex curve \mbox{$z\mapsto \lambda_\m(z)$} ($z\in \C$)  in $i\Z^n \backslash \C^n$ is the image of a complex line in $\C^n$ with slope given by $\m$, and hence its real part  is a straight line in $\R^n$ through the origin with slope  $\m$, i.e., $t\mapsto (m_1 t, \cdots, m_n t)$, $t\in \R$, and  $\lambda_\m(z)$ is contained  in $\mathrm{Span}_\C(\sigma)$. 

      By the definition of the topology of $\overline{T}_\Sigma$ above, $\lim_{z\to 0} \lambda_\m(z)$ converges to the distinguished point $0_\Sigma$ in $O(\sigma)$, i.e., to the image of the origin of $\C^n$  in $O(\sigma)$. 
      This proves Lemma \ref{origin}. 
\end{proof}

\begin{lem}\label{coincide-special-point}
      For any cone $\sigma \in \Sigma$, a sequence $\z_j$ in $(\C^*)^n=T$ converges to the distinguished point $x_\sigma$ in  the toric variety $X_\Sigma$ if and only if it converges to the distinguished point $0_\sigma$ in the topological model $\overline{T}_\Sigma$ 
\end{lem}

\begin{proof}
      We note that for the open subset $U_\sigma \subset X_\Sigma$, under the embedding of $U_\sigma \subset \C^k$ in Equation \ref{affine-toric} on page \pageref{affine-toric}, the coordinates $\t^{\m_i}$ of the distinguished point $x_\sigma$ are either 0 or 1 depending on whether the element $\m_i$ in $\sigma^\vee\cap \Z^n$ is zero or positive on $\sigma$. This implies that a sequence $\z_j \in (\C^*)^n$ converges to the distinguished point $x_\sigma$ if and only if the following conditions are satisfied:
      \begin{enumerate}
	  \item for any $\m\in \sigma^\vee\cap \Z^n$ with $\m|_\sigma >0$ it holds $\z_j^\m \to 0$ as $j\to +\infty$, 
	  \item for any  $\m\in \sigma^\vee\cap \Z^n$ with $\m|_\sigma =0$ it holds $\z_j^\m \to 1$ as $j\to +\infty$.
      \end{enumerate}

      Note that the vectors in $ \sigma^\vee\cap \Z^n$ with $\m|_\sigma \geq 0$ span the dual cone $\sigma^\vee$, i.e., linear combinations of these vectors with nonpositive coefficients give $\sigma^\vee$. In terms of the identification $(\C^*)^n= i \Z^n \backslash \C^n$, write $\z_j=\x_j + i\y_j$ with $\x_j, \y_j \in \R^n$ as in the definition of the topology of $\overline{T}_\Sigma$,  then the above condition on $\z_j$ is equivalent to the following conditions:
      \begin{enumerate}
	  \item The real part $\x_j$ can be written as  $\x_j=\x_j'+\x_j''$ such that when $j\to +\infty$,
	  \begin{enumerate}
	        \item the first part $\x_j'$ is contained in the interior of the cone $\sigma$ and its distance to the relative  boundary of $\sigma$ %(i.e., the boundary of $\sigma$ in the linear subspace spanned by $\sigma$)
	        goes to infinity,
	        \item the second part $\x_j''$ is bounded. 
	  \end{enumerate}
	  \item The image of $\z_j$ in $O(\sigma)=i\Z^n\backslash \C^n /\mathrm{Span}_\C(\sigma)$ under the projection 
	  \[
	        i\Z^n\backslash \C^n\to i\Z^n\backslash \C^n  /\mathrm{Span}_\C(\sigma)
	  \]
	  converges to the image in $O(\sigma)$ of the zero vector in $\C^n$.
      \end{enumerate} 

      By the definition of $\overline{T}_\Sigma$, this is exactly the conditions for the sequence $\z_j$ to converge to the distinguished point $0_\sigma$ in $\overline{T}_\Sigma$.  This proves Lemma \ref{coincide-special-point}. 
 \end{proof}

\noindent{\em Proof of Proposition \ref{identification3}.} \label{proof_prop_identification}

 The idea of the proof  is to use the continuous actions of $T$ on $X_\Sigma$ and $\overline{T}_\Sigma$ to extend the equivalence of convergence of interior sequences
 to the distinguished point $x_\sigma=0_\sigma$  in Lemma \ref{coincide-special-point} to other boundary points.

 Under the action of $T$, the orbit  $T\cdot 0_\sigma$ in $\overline{T}_\Sigma$ gives $O(\sigma)$. As pointed out in Proposition \ref{orbits} on page \pageref{orbits} and Proposition \ref{distinguished_point}, the orbit $T\cdot x_\sigma$ in $X_\Sigma$ gives the orbit corresponding to $\sigma$. It can be seen that the stabilizer of the distinguished point $x_\sigma \in \orb(\sigma)$ in $T= i\Z^n \backslash \C^n$ is equal to the subgroup $i\mathrm{Span}_\C(\sigma)\cap \Z^n\backslash \mathrm{Span}_\C(\sigma)$ (see \cite[Lemma 3.2.5]{cls}).
 By the definition of $\overline{T}_\Sigma$, the stabilizer of  the point $o_\sigma \in O(\sigma)$ is also equal to $i\mathrm{Span}_\C(\sigma)\cap \Z^n\backslash \mathrm{Span}_\C(\sigma)$.
 Therefore, there is a canonical identification between $\orb(\sigma)$ and $O(\sigma)$. 

 By Lemma \ref{coincide-special-point}, for any sequence $\z_j$ in $T$, $\z_j \to x_\sigma$  in $X_\Sigma$ if and only if $\z_j \to 0_\sigma$  in $\overline{T}_\Sigma$. Take any such sequence $\z_j \in (\C^*)^n$  with $\lim_{j\to +\infty} \z_j=x_\sigma$. Let   $\t_j \in (\C^*)^n$ be any converging sequence with $\lim_{j\to\infty} \t_j=\t_\infty$.
 For both the toric variety $X_\Sigma$ and the bordification $\overline{T}_\Sigma$, the continuous actions of $T$ on $X_\Sigma$ and $\overline{T}_\Sigma$ in Proposition \ref{extends}  imply that the sequence $\t_j \z_j$ converges to $\t_\infty \cdot x_\sigma$ in $X_\Sigma$, and to $\t_\infty \cdot 0_\sigma \in O(\sigma)$ in $\overline{T}_\Sigma$ respectively.
 This implies that a sequence of interior points $\z_j$ in $T$ converges to a boundary point in the orbit $\orb(\sigma)\subset X_\Sigma$ if and only if it converges to a corresponding point in $O(\sigma)\subset \overline{T}_\Sigma$.
 Since $\sigma$ is an arbitrary cone in $\Sigma$  and $\t_j$ is an arbitrary convergent sequence in $(\C^*)^n$, this  proves the topological description of toric varieties in Proposition \ref{identification3}.

%\vspace{.2in}
%\noindent{\bf Fans coming from convex polytopes}.
%\vspace{.1in}

\subsection{Fans coming from convex polytopes}

 One way to construct fans in $\R^n$ is to start with a rational convex polytope $P$ which contains the origin as an interior point. As a more detailed reference for this construction see \cite[Section 1.5]{fu}.

 Recall that a convex polytope $P$ in $\R^n$ is  the convex hull of finitely many points of $\R^n$. If the vertices of $P$ are contained in $\Z^n$, then $P$ is called a {\em rational convex polytope}\footnote{This definition is due to \cite[p. 24]{fu}. At some other places, $P$ is called  a rational convex polytope if the vertices of $P$ are contained in $\Q^n$, and  $P$ is called an integral convex polytope if the vertices of $P$ are contained in $\Z^n$.}.

 Assume that $P$ is a rational convex polytope in $\R^n$ and contains the origin as an interior point. Then each face $F$ of $P$ spans a rational  polyhedral cone 
 \[
      \sigma_F=\R_{\geq 0}\cdot F,
 \] 
 i.e., the face $F$ is a section of the cone $\sigma_F$,  and these cones $\sigma_F$ form a fan in $\R^n$, denoted by $\Sigma_P$. See for example Figure \ref{pic:fan} below. Denote the toric variety defined by the fan $\Sigma_P$ by $X_{\Sigma_P}$. Since the support of the fan $\Sigma_P$ is equal to $\R^n$, $X_{\Sigma_P}$ is compact.
 Note that for any integer $k$, the scaled polytope $kP$ is also a rational polytope and gives the same fan, $\Sigma_{kP}=\Sigma_P$. 

%\vspace{.3in}
     % \input{Bild_fan_2D}
%\vspace{.3in}

\begin{figure}[h!]
 \centering
 \includegraphics{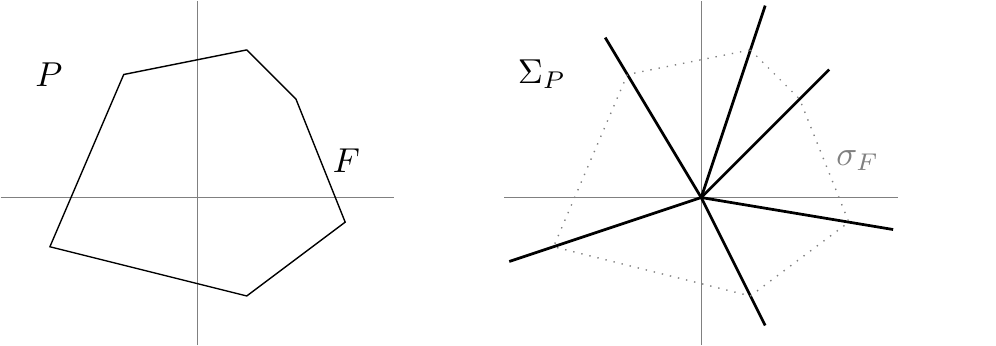}
 \caption{A rational convex polytope $P$ and its corresponding fan $\Sigma_P$ in $\R^2$} \label{pic:fan}
\end{figure}

 It is known that not every fan $\Sigma$  in $\R^n$ comes from such a rational convex polytope $P$, as the following example shows. 
 \begin{ex}\cite[p.25]{fu}
      Take the fan generated by the eight halflines through the origin and one of the following eight points: 
      \[
	  (-1, \pm 1, \pm 1), (1, -1, \pm 1), (1,  1, -1), (1,2,3).
      \]
      Then it is not possible to find eight points, one on each of the halflines, such that for each of the six cones the four corresponding generating points lie on one affine hyperplane.  
 \end{ex}

 The toric varieties defined by fans $\Sigma_P$ which do come from  rational convex polytopes $P$ as above have a simple characterization: 

 \begin{prop} \label{polytope-toric}
      A toric variety $X_\Sigma$ is a projective variety if and only if the fan $\Sigma$ is equal to the fan $\Sigma_P$ induced from a rational convex polytope $P$ containing the origin as an interior point as above.
 \end{prop}

 In \cite[p. 26]{fu} and \cite[p. 219]{cox}, a rational polytope dual to $P$ is used to construct a toric variety. 
 
 \begin{defi}
      The {\em polar set} $P^\circ$  of a convex polytope $P$ is defined  by
      \begin{equation}\label{polar}
	  P^\circ=\{v\in \R^n\mid \langle v, u\rangle \geq -1, \text{ for all \ } u \in P \}.
      \end{equation}
 \end{defi}

 \begin{rem} \ 
      \begin{itemize}
       \item When $P$ is a rational convex polytope containing the origin as an interior point, then $P^\circ$ is also a rational convex polytope containing the origin as an interior point.
       \item When $P$ is symmetric with respect to the origin, then $P^\circ$ is equivalent to another common definition of polar set: 
       \[
	  \{v\in \R^n\mid \langle v, u\rangle \leq 1, \text{ for all \ } u \in P\}.
       \]
      \end{itemize}

 \end{rem}

 The following fact is well-known, for a reference see for example \cite[p. 24]{fu} or \cite[Lemma 3.7]{hsw} for a proof. 

\begin{prop}\label{duality} 
      There is a duality between $P$ and $P^\circ$ which is given by an one-to-one correspondence between the set of faces of $P$ and the set of faces of $P^\circ$ which reverses the inclusion relation.
\end{prop}

The correspondence is as follows: Let $F$ be a face of $P$. Then there is exactly one face $F^\circ$ of $P^\circ$, called the {\em dual face of $F$}, which satisfies the following two conditions:
\begin{enumerate}
      \item for any $x\in F$ and $y\in F^\circ$ it holds: $\langle x, y\rangle= -1$,
      \item $\dim(F) + \dim(F^\circ) = n-1$. 
\end{enumerate}

 \begin{proof}[Proof of Proposition \ref{polytope-toric}]\ \\
      The fan $\Sigma_{P^0}$ associated with the dual polytope $P^\circ$ above is called the {\em normal fan}  of the polytope $P$ and is defined for example in \cite[pp. 217-218]{cox} or  \cite[Proposition, p. 26]{fu}. The toric varieties there are defined by the normal fans $\Sigma_{P^0}$ of $P$ and not by the fans $\Sigma = \Sigma_P$ as in this paper. But since the polar set of $P^\circ$ is equal to $P$, $(P^\circ)^\circ = P$, the above statement in Proposition \ref{polytope-toric} is equivalent to that in \cite[Theorem 12.2]{cox} which is stated in terms of normal fans of rational polytopes. 
 \end{proof}

%\vspace{.2in}
%\noindent{\bf Real and nonegative part of toric varieties and the moment map.}
%\vspace{.1in}

\subsection{Real and nonnegative part of toric varieties and the moment map}

 For every toric variety $X_\Sigma$, there is the notion of the nonnegative part $X_{\Sigma, \geq 0}$ (see also \cite[p. 78]{fu}, \cite[\S 1.3]{od1} and \cite[\S 6]{so}).

 In $\C^*$, the real part is $\R^*=\R_{> 0} \cup \R_{< 0}$, and in the complex torus $(\C^*)^n$, the real part is $(\R^*)^n$, which has $2^n$-connected components. The positive part of $\C^*$ is $\R_{>0}$, and the positive part of $(\C^*)^n$  is $(\R_{>0})^n$.

 Under the identification (Equation \ref{identification2} on page \pageref{identification2})
 \[
      (\C^*)^n\cong i\Z^n\backslash \C^n= \R^n \times i\Z^n\backslash \R^n
 \] 
the positive part $(\R_{>0})^n$ corresponds to $\R^n \times {i 0}\cong \R^n$.

 \begin{defi} \cite[Definition 6.2]{so}
      For any toric variety $X_\Sigma$, the closure of the positive part $(\R_{>0})^n$ is called the {\em nonnegative part} of $X_\Sigma$, denoted by $X_{\Sigma, \geq 0 }$. 
 \end{defi}

 Under the identification in Proposition \ref{identification3}, $X_{\Sigma, \geq 0}$ can be described as follows:

 \begin{prop}\label{real-orbit}
      For any fan $\Sigma$ of $\R^n$, the nonnegative part $X_{\Sigma, \geq 0}$ is homeomorphic to the space 
      \begin{equation*}
	  \overline{\R^n}_\Sigma=\R^n \cup \coprod_{\substack{\sigma \in \Sigma, \\ \sigma\neq \{0\}}} \R^n/\mathrm{Span}_\R(\sigma)
      \end{equation*}
      with the following topology: An  unbounded sequence $\x_j \in \R^n$ converges to a boundary point $\x_\infty$ in $\R^n/\mathrm{Span}_\R(\sigma)$ for a cone $\sigma$ if and only if one can write $\x_j=\x_j'+ \x_j''$ such that the following conditions are satisfied:
      \begin{enumerate}
	  \item when $j\to +\infty$, $\x_j'$ is contained in the cone $\sigma$ and its distance to the relative boundary of $\sigma$ goes to infinity, \item $\x_j''$ is bounded,
	  \item the image of $\x_j$ in  $\R^n/\mathrm{Span}_\R(\sigma)$ under the projection $\R^n \to \R^n/\mathrm{Span}_\R(\sigma)$ converges to $\x_\infty$.
      \end{enumerate}
 \end{prop}

 This proposition was explained in detail and proved in \cite[pp. 2-6]{am} and motivated Proposition \ref{identification3} above.

 If we denote $\R^n/\mathrm{Span}_\R(\sigma)$ by $O_\R(\sigma)$, then for any two cones $\sigma_1, \sigma_2$, it holds that $O_\R(\sigma_1)$ is contained in the closure of $O_\R(\sigma_2)$ if and only if $\sigma_2$ is a face of $\sigma_1$. Therefore, the nonnegative part $X_{\Sigma, \geq 0}$ can be rewritten as \begin{equation}\label{real-orbit-2}
      X_{\Sigma, \geq 0}= \R^n \cup \coprod_{\substack{\sigma \in \Sigma, \\ \sigma\neq \{0\}}} O_\R(\sigma).
 \end{equation}
 Consequently, we have 
 \begin{cor} 
      The translation action of $\R^n$ on $\R^n$ extends to a continuous action on $X_{\Sigma, \geq 0}$, and the  decomposition of  $X_{\Sigma, \geq 0}$ in Equation \ref{real-orbit-2} is the decomposition  into $\R^n$-orbits.  This decomposition of the nonnegative part of the toric variety $X_{\Sigma, \geq 0}$  is a cell complex dual to the fan $\Sigma$. If $\Sigma=\Sigma_P$ for a rational convex polytope $P$ containing the origin as an interior point, then this cell complex structure is isomorphic to the cell structure of the polar set $P^\circ$. 
 \end{cor}

 Using the moment map for projective toric varieties,  we can realize this cell complex of $X_{\Sigma, \geq 0}$ and hence the compactification $\overline{\R^n}_\Sigma$  by a bounded  convex polytope: 

 Let $P$ be a rational convex polytope containing the origin as an interior point, and $X_{\Sigma_P}$ the associated projective variety.  By definition, each cone $\sigma$ of $\Sigma_P$ corresponds to a unique face $F_\sigma$ of $P$, which gives by Proposition \ref{duality} a dual face $F^\circ_\sigma$ of the polar set  $P^\circ$. 
 %Then the result on the moment map $\mu$ of $X_{\Sigma_P}$ \cite[p. 94]{od1} \cite[\S 4.2]{fu} \cite[\S 8]{so} \cite[Theorem 1.2]{js} can be stated as follows.
 
 \begin{prop}
      The moment map induces a homeomorphism
      \begin{equation*}\label{moment}
	  \mu: X_{\Sigma_P, \geq 0}\to P^\circ
      \end{equation*} 
      such that for every cone $\sigma \in \Sigma_P$, the positive part of the orbit $O(\sigma)$ as a complex torus, or equivalently  the orbit $O_\R(\sigma)$ in Proposition \ref{real-orbit},  is mapped  homeomorphically to the relative interior of the face $F^\circ_\sigma$ corresponding to the cone $\sigma$. 
 \end{prop}

 For more details about the moment map and the induced homeomorphism see \cite[p. 94]{od1}, \cite[\S 4.2]{fu}, \cite[\S 8]{so} and \cite[Theorem 1.2]{js}.

\section{Polyhedral metrics}\label{poly}

 In this section, we recall the definition of polyhedral norms on $\R^n$. They are rather  special in view of the   Minkowski  geometry of  normed real vector spaces $\R^n$ and the Hilbert geometry  of bounded convex subsets of real vector spaces.

 Let $|\!| \cdot |\!|$ be  an  {\em asymmetric norm} on $\R^n$, i.e., a function $|\!| \cdot |\!|: \R^n \to \R_{\geq 0}$ satisfying:
 \begin{enumerate}
      \item For any $x\in \R^n$, if $|\!| x |\!|=0$, then $x=0$.
      \item For any $\alpha \geq 0$ and $x\in \R^n$, $|\!| \alpha x |\!|=\alpha |\!| x |\!|.$
      \item For any two vectors $x, y\in \R^n$, $ |\!| x + y|\!| \leq |\!| x |\!|+ |\!| y |\!|$.
 \end{enumerate}

 In particular, $|\!| x |\!|$ and $|\!| -x |\!|$ may not be equal to each other. If the second condition is replaced by the stronger condition: $|\!| \alpha x |\!|=|\alpha| |\!| x |\!|$ for all $\alpha \in \R$, then $|\!| \cdot |\!|$ is symmetric and is a usual norm on $\R^n$.

 Normed vector spaces have been extensively studied. They are also called Minkowski geometry in \cite{th}.
 Asymmetric norms on vector spaces have also been studied systematically, see \cite{cob}. In terms of their connection with convex domains below, they are natural. 

 Given an asymmetric  norm on $\R^n$, the unit ball $B_{|\!| \cdot |\!|}$ of  $|\!| \cdot |\!|$,
 \[
      B_{|\!| \cdot |\!|}=\{x\in \R^n\mid |\!| x|\!| \leq 1\},
 \]
 is a closed convex subset of $\R^n$ which contains the origin as an interior point. Conversely, given any convex closed subset $P$ of $\R^n$ which contains the origin as an interior point, we can define the {\em Minkowski functional} on $\R^n$ by 
 \begin{equation*}
      |\!| x|\!|_P=\inf\{\lambda >0\mid x\in \lambda P\}.
 \end{equation*}

 It can be checked easily that $|\!| \cdot |\!|_P$ defines an asymmetric norm on $\R^n$. If $P$ is symmetric with respect to the origin, i.e., $-P=P$, then $|\!| x|\!|_P$ is a norm on $\R^n$. 

 It is also easy to see that the unit ball of $|\!| \cdot |\!|_P$ is equal to $P$. Since any asymmetric norm $|\!| \cdot|\!|$ on $\R^n$ is uniquely determined by its unit ball, it is of the form $|\!| \cdot|\!|_P$ for some closed convex domain $P$ in $\R^n$ containing the origin in its interior. 

 \begin{defi}
      When $P$ is a polytope, the asymmetric norm $|\!| \cdot |\!|_P$ is called a {\em polyhedral norm}. If $P$ is a rational polytope with respect to the integral structure $\Z^n \subset \R^n$, the norm $|\!| \cdot |\!|_P$ is also called a {\em rational polyhedral norm}.
 \end{defi}

 \begin{rem}[Connections to the Minkowski and Hilbert geometry]
      This interplay between convex subsets of $\R^n$ and norms on $\R^n$ plays a foundational role in the convex analysis  of Minkowski geometry, see for example \cite{gru} and \cite{th}. If the lattice $\Z^n \subset \R^n$ is taken into account, connections with number theory and counting of lattices points are established and the structure becomes richer. The geometry of numbers relies crucially on these connections, see also \cite{grl} and \cite{ba}.

      There is another metric space associated with a convex domain $\Omega$ of $\R^n$. It is the domain $\Omega$ itself equipped with the Hilbert metric defined on it. When $\Omega$ is the unit ball of $\R^2$, this is the Klein's model of the hyperbolic plane.  In general, the Hilbert metric is a complete metric on  $\Omega$ defined through the cross-ratio. See \cite{del} for details. Since $\Omega$ is diffeomorphic to $\R^n$, the Hilbert metric induces a metric on $\R^n$. 

      When $\Omega$ is the interior of a convex polytope $P$,  the Hilbert metric on $\Omega$ is quasi-isometric to a polyhedral norm \cite{be} \cite{ve}. The polyhedral Hilbert metric associated with a polytope  $P$  is isometric to a normed vector space if and only if the polytope $P$ is the simplex \cite[Theorem 2]{fk}. Furthermore,  polyhedral Hilbert metrics  have also special isometry groups \cite{lw}. See also \cite{ln} for other special properties of these Hilbert metrics.

      These discussions show that polyhedral norms on $\R^n$, in particular rational polyhedral norms, are very special in the context of the Minkowski  geometry \cite{th} and the Hilbert geometry \cite{del}.
\end{rem}

\section{Horofunction compactification of metric spaces}\label{horo-section}

 In this section we recall briefly the horofunction compactification of metric spaces, which was first introduced in \cite[\S 1.2]{gr}.

 Let $(X, d)$ be a proper metric space, for example, a locally compact metric space. The metric $d(\cdot, \cdot)$ can be asymmetric, i.e., it satisfies all conditions of a usual metric except for the symmetry: $d(x, y) \neq d(y, x)$ possible. Such metrics arise naturally in view of polyhedral norms on $\R^n$ as we saw in the previous subsection.

 Let $C(X)$ be the space of continuous functions on $X$ with the compact-open topology. Let $\widetilde{C}(X)$ be the quotient space $C(X)/\{ \text{constant functions}\}$. Denote the image of a function $f\in C(X)$ in $\widetilde{C}(X)$ by $[f]$.

 Define a map 
 \begin{equation*}
      \psi: X \to \widetilde{C}(X), \quad x\mapsto [d(\cdot, x)].
 \end{equation*}

 If we fix a basepoint $x_0\in X$, then we can consider functions normalized to take value 0 at $x_0$, and get a map
 \begin{equation*}
      \psi: X \to C(X), \quad x \mapsto d(\cdot, x)-d(x_0, x).
 \end{equation*}

 It can be shown (see \cite[\S 3]{bgs}) that the closure of $\psi(X)$ in $X$ is compact and that if $X$ is a geodesic metric space and $d$ is symmetric with respect to convergence\footnote{Symmetric with respect to convergence means that for any sequence $(x_n)_n$ and $x \in X$ it holds: $d(x_n, x) \lora 0$ if and only if $d(x, x_n) \lora 0$.}, then $\psi$ is an embedding. % and  \cite[Chap. 3]{sch}). 
 This compact closure is called the {\em horofunction compactification} of $X$ and denoted by $\overline{X}^{hor}$. Functions in the boundary $\overline{X}^{hor} \setminus X$ are called {\em horofunctions}  of $X$. Limits of points $\psi(x_j)$  in the horofunction compactification $\overline{X}^{hor}$  when $x_j$ moves along an almost geodesic are called  {\em Busemann functions}. 

 \begin{rem}
      There are other embeddings of compact Riemann manifolds into function spaces, in particular using the heat kernel in \cite{bbg} \cite{kk}. See also \cite[\S 1.1]{gr2}. In comparison, the horofunction compactification is simple and direct and applies to all metric spaces which are not necessarily Riemannian manifolds. It also establishes a connection between the metric and function properties of the space. 
 \end{rem}
 
 \begin{ex}
      When $X$ is the Euclidean space  $\R^n$ with the standard Euclidean metric, it can be shown that horofunctions are linear functions of the form $-\langle \cdot, \mathbf v\rangle$, where $\mathbf v\in \R^n$ is a unit vector, and this horofunction is the Busemenn function corresponding to the geodesic $t \mathbf v$, $t\in \R$. Consequently, the horofunction compactification $\overline{\R^n}^{hor}$ is homeomorphic to the geodesic (or visual) compactification of $\R^n$ which is obtained by adding the unit sphere at infinity.
 \end{ex}

 \begin{rem}[Horofunction compactification of certain spaces]
      When $X$ is a simply connected complete Riemannian manifold of nonpositve curvature, the compactification $\overline{X}^{hor}$ was also shown in \cite{bgs} to be homeomorphic to the geodesic compactification, whose boundary is the set of equivalence classes of geodesics.
      More generally, when $X$ is a complete CAT(0)-space, the same result holds  \cite[Theorem 8. 13]{bh}. In these cases, every horofunction is a {  Busemann function}. When $X$ is a symmetric space of noncompact type, horofunctions can be computed explicitly and they are related to the horospherical coordinates of $X$ with respect to parabolic subgroups and hence to functions in the harmonic analysis on symmetric spaces (see \cite{gjt} or \cite{ha}).

      If a metric space is not a CAT(0)-space, then it is in general difficult to determine its horofunction compactification. One special class of contractible metric spaces which are not  CAT(0)-spaces are $\R^n$ with polyhedral norms. Horofunctions of $\R^n$ with respect to polyhedral norms were computed in \cite{kmn}, \cite{wa1} and \cite{js}. See the next section for more detail.
      
      The Hilbert metric on the interior of a  polytope $P$  in $\R^n$ also gives  a contractible metric space which is not a CAT(0)-space either. The horofunctions of such metric spaces were computed in \cite{wa3} and \cite{wa4}.
 \end{rem}

 Besides the application to the noncommutative geometry  mentioned in the introduction, another important application of horofunctions is that they give rise to horoballs of the boundary points. They are special subsets attached to points at infinity and  have various applications: 
 \begin{enumerate}
      \item The geometry of nonpositively curved manifolds, in particular the structure of the thin part and compactifications of nonpositively curved Riemannian manifolds of finite volume \cite{bgs}. This application was the motivation to study the horofunctions in \cite{bgs}.
      \item The dynamics of certain classes of nonlinear self-maps of convex cones   \cite{kmn}.
      \item The multiplicative ergodic theory and the law of large numbers for random walks \cite[Theorem 1.1]{kl}.
\end{enumerate}

\section{Horofunction compactification of polyhedral normed spaces}
% and proof of Theorem \ref{main}}

 In this section, we recall results on the horofunction compactification of $\R^n$ with respect to a polyhedral norm given in \cite{js}. Together with the characterization of the nonnegative part of the toric varieties in Proposition \ref{real-orbit}, it will be used to prove Theorem \ref{main}.

 Let $P$ be a convex polytope of $\R^n$ which contains the origin as an interior point.  Let $|\!|\cdot |\!|_P$ be the polyhedral norm on $\R^n$ whose unit ball is equal to $P$. This defines an asymmetric metric on $\R^n$ by
 \[
      d_P(\x, \y)=|\!|\y-\x |\!|_P.
 \]
 Since $|\!|\cdot |\!|_P$ is an asymmetric norm,  $d_P(\x, \y)\neq d_P(\y, \x)$ in general. 

 For each face $F$ of $P$, let $\sigma_F=\R_{\geq 0}\cdot F$ be the cone in $\R^n$ spanned by $F$,  and $V(F)$ the linear subspace spanned by $\sigma_F$. 
 Let 
 \[
      \Pi_F: \R^n \to V(F)
 \] 
 denote the orthogonal projection onto the subspace $V(F)$.

 The result  \cite[Theorem 3.10]{js} on the horofunction compactification of $(\R^n, d_P(\cdot, \cdot))$ can be stated as follows.

 \begin{prop}\label{horo}
      Let $|\!|\cdot |\!|_P$ be a polyhedral norm on $\R^n$  as above. Then an unbounded sequence of points $\x_j$ in $\R^n$ converges to a boundary point in the horofunction compactification $\overline{\R^n}^{hor}$ of $(\R^n, d_P(\cdot, \cdot))$ if and only if there exists a proper face $F$ of $P$ and a point $\x_\infty \in \R^n / V(F)$  such that the following conditions are satisfied:
      \begin{enumerate}
	  \item When $j\gg 1$, $\Pi_F(\x_j)$ is contained in the cone $\sigma_F$.
	  \item Let $\partial_{\mathrm{rel}}\sigma_F$ be the relative boundary of the cone $\sigma_F$, i.e., the boundary points of $\sigma_F$ in $V(F)$. Then the distance $d(\Pi_F(\x_j), \partial_{\mathrm{rel}}\sigma_F)\to +\infty$ as $j\to +\infty$.
	  \item $\x_j \to \x_\infty$ in $\R^n / V(F)$ as $j\to +\infty$.  %\mnote{is the third condition correct like that?}
      \end{enumerate}
 \end{prop}

 The horofunctions of $(\R^n, d_P(\cdot, \cdot))$  can be described explicitly  in terms of convergent sequences in the above proposition.

 Recall that  $P^\circ$ is the polar set of $P$, which is also a polytope containing the origin as an interior point. For a convex set $C \subset (\R^n)^*$ we can define a {\em pseudo-norm}:
 \[
      \lvert x \rvert_C \coloneqq -\inf_{q \in C}\langle q, x \rangle.
\]

 Using this pseudo-norm, we define functions for each face $E \subset P^\circ$ and $\p \in \R^n$:
 \begin{align*}
      h_{E,\p}: \R^n &\lora \R,\\
      \y &\longmapsto \lvert \p-\y \rvert_E - \lvert \p \rvert_E.   
 \end{align*}

 \begin{lem}[\cite{js}, Lemma 2.16] \label{h_orthogonal}
      Let $F$ be the face of $P$ dual to the face $E$ of $P^\circ$. Then the function $h_{E, \p}$ only depends on the image of $\p$ in $\R^n/ V(F)$.
 \end{lem}
 
 Then the horofunction compactification of $\R^n$ with respect to the polyhedral norm $|\!|\cdot |\!|_P$ in Proposition \ref{horo} can be described more explicitly as follows. More details are also provided in \cite[Theorem 3.10]{js}. %\cite[Theorem 2.11]{js}.

 \begin{prop}
      Under the conditions and notations in Proposition \ref{horo}, let $E$ be the face of the polar set $P^\circ$ which is dual to $F$, and $\x_j$ a convergent sequence  in $\overline{\R^n}^{hor}$ defined there. 
      Then $\x_j$ converges to the function $h_{E, \x_\infty}$, i.e., the function $d_P(\z, \x_j)- d_P(0, \x_j)$ converges to $h_{E, \x_\infty}(\z)$ uniformly over compact subsets of $\R^n$.
 \end{prop}
 
 Note that $h_{E, \x_\infty}$ is well-defined by Lemma \ref{h_orthogonal}.  The paper \cite{js} relies on the computation of horofunctions of normed vector spaces $\R^n$ in \cite{wa1} (see also \cite[Theorem 3.6]{js}). One basic ingredient is  that the pseudo-norm with respect to $P^\circ$ is the polyhedral norm with respect to $P$: 
 \[
      \lvert \cdot \rvert_{P^\circ} = \norm_P.
 \] 
 Therefore $d_P(\z, \x_j)- d_P(0, \x_j) = \lvert \x_j - \z \rvert_{P^\circ} - \lvert \x_j \rvert_{P^\circ} = h_{P^\circ, \x_j}$ and so it is reasonable to expect that when a sequence of points $\x_j$ in $\R^n$ goes to the boundary (or infinity)  as in  Proposition \ref{horo}, the limit of $d_P(\z, \x_j)- d_P(0, \x_j)$ is related to $h_{E, \x_\infty}$ for some face  $E$ of $P^\circ$. The detailed computation is  crucial to Proposition \ref{horo} and worked out in the proof of Theorem 3.10 in \cite{js}.

 \begin{rem}
      Horofunctions of $\R^n$ with respect to polyhedral norms are also computed in \cite[Theorem 4.2]{kmn}, or more precisely, horoballs are computed there. They are not expressed in terms of the $h_{E, \x}$ described above. 
 \end{rem}

 \begin{proof}[Proof of Theorem \ref{main}]\ \\
      Let $P$ be a rational polytope containing the origin as an interior point.  Let $\Sigma=\Sigma_P$ be the fan obtained by taking cones over the faces of $P$. 
      By Propositions \ref{real-orbit} and \ref{horo}, an unbounded sequence  of $\R^n$ converges to a boundary point in the compactification $\overline{\R^n}_\Sigma$  if and only if it converges in the horofunction compactification  $\overline{\R^n}^{hor}$ with respect to the polyhedral norm $|\!| \cdot |\!|_P$ or rather the metric $d_P$.  Therefore, the two compactifications  $\overline{\R^n}_\Sigma$ and $\overline{\R^n}^{hor}$  of $\R^n$ are homeomorphic.
      Then  Proposition  \ref{real-orbit} again  implies that the nonnegative part $X_{\Sigma, \geq 0}$  of the toric variety is homeomorphic to the horofunction compactification  $\overline{\R^n}^{hor}$. 
 \end{proof}
 
 \begin{rem}\label{norbert}
      Given the one-to-one correspondence between the toric varieties and rational polyhedral norms in Theorem \ref{main} and the fact that each polytope $P$ also determines a Hilbert metric $d_H(\cdot, \cdot)$ on the  interior $\text{int}(P)$ of $P$, one natural question is whether there exists a similar relation between $X_{\Sigma_P, \geq 0}$ and the horofunction compactification of $(\text{int}(P), d_H(\cdot, \cdot))$. The results in \cite{wa4} and \cite{wa5} show that besides the Hilbert metric, Funk metric and reverse Funk metric should also be considered, and  that the horofunction compactifications of the Funk metric seems to be related to  $X_{\Sigma_{P^\circ}, \geq 0}$, the toric variety associated to the polar set $P^0$, and the horofunction compactification of the Hilbert metric is more complicated. 
      This question will be treated elsewhere.
\end{rem}

 \noindent{\em Acknowledgments:}
 We would like to thank Norbert A'Campo for his interest in the problem considered in this paper and raising Question \ref{norbert} about the Hilbert geometry.


\begin{thebibliography}{WX}

\bibitem[AGW]{agw}M. Akian, S.  Gaubert, C. Walsh, 
The max-plus Martin boundary. 
Doc. Math. 14 (2009), 195--240. 


\bibitem[An]{an}P.D. Andreev, Ideal closures of Busemann space and singular Minkowski space. Preprint,
2004. Eprint math.GT/0405121



\bibitem[AM]{am}A. Ash, D. Mumford, M. Rapoport, 
Y.-S. Tai,  {\em Smooth compactifications of locally symmetric varieties}. Second edition. With the collaboration of Peter Scholze. Cambridge Mathematical Library. Cambridge University Press, Cambridge, 2010. x+230 pp.


\bibitem[AT]{at}P. Alvarez Paiva,  A. Thompson, 
 Volumes on normed and Finsler spaces. {\em A sampler of Riemann-Finsler geometry}, 
1--48, Math. Sci. Res. Inst. Publ., 50, Cambridge Univ. Press, Cambridge, 2004.

\bibitem[BGS]{bgs}W. Ballmann, M. Gromov, V. Schroeder, {\em Manifolds of nonpositive curvature.} 
Progress in Mathematics, 61. Birkh\"auser Boston, Inc., Boston, MA, 1985. vi+263 pp.

\bibitem[Ba]{ba}A. Barvinok, 
{\em A course in convexity. }
Graduate Studies in Mathematics, 54. American Mathematical Society, Providence, RI, 2002. x+366 pp. 

\bibitem[BBG]{bbg}P. B\'erard, G.  Besson, S.  Gallot, 
{\em Embedding Riemannian manifolds by their heat kernel.} 
Geom. Funct. Anal. 4 (1994), no. 4, 373--398.


\bibitem[Be]{be}A. Bernig,  Hilbert geometry of polytopes. Arch. Math. (Basel) 92 (2009), no. 4, 314--324. 

\bibitem[BH]{bh}M. Bridson, A.  Haefliger, 
{\em Metric spaces of non-positive curvature.} 
Grundlehren der Mathematischen Wissenschaften, 
319. Springer-Verlag, Berlin, 1999. xxii+643 pp.


\bibitem[Cob]{cob}S. Cobzas, 
{\em Functional analysis in asymmetric normed spaces.} 
Frontiers in Mathematics. Birkh\"auser/Springer Basel AG, Basel, 2013. x+219 pp.


\bibitem[CVV]{cvv}B. Colbois, C. Vernicos, P.  Verovic,  Hilbert geometry for convex polygonal domains. J. Geom. 100 (2011), no. 1-2, 37--64.



\bibitem[Co]{co}A. Connes,  Compact metric spaces, Fredholm modules, and hyperfiniteness. Ergodic Theory Dynam. Systems 9 (1989), no. 2, 207--220.

\bibitem[Cox]{cox}D. Cox, What is a toric variety? 
{\em Topics in algebraic geometry and geometric modeling}, 
203--223, Contemp. Math., 334, Amer. Math. Soc., Providence, RI, 2003.

\bibitem[CLS]{cls}D. Cox, J. Little,  H. Schenck, {\em  Toric varieties.}
 Graduate Studies in Mathematics, 124. American Mathematical Society, Providence, RI, 2011. xxiv+841 pp. 

\bibitem[DFS]{dfs}F. Dal'bo,   M. Peign\'e,  A. Sambusetti, 
 On the horoboundary and the geometry of rays of negatively curved manifolds. Pacific J. Math. 259 (2012), no. 1, 55--100.



\bibitem[deL]{del}P. de la Harpe,  On Hilbert's metric for simplices. {\em Geometric group theory}, 
Vol. 1 (Sussex, 1991), 97--119, London Math. Soc. Lecture Note Ser., 181, Cambridge Univ. Press, Cambridge, 1993.

\bibitem[De]{de}M. Develin, 
Cayley compactifications of abelian groups. 
Ann. Comb. 6 (2002), no. 3-4, 295--312. 






%\bibitem[DS]{ds}M. Deza, M.D. Sikiric, Voronoi Polytopes for Polyhedral Norms on Lattices, arXiv:1401.0040.

\bibitem[FK]{fk}T. Foertsch, A. Karlsson, 
Hilbert metrics and Minkowski norms. J. Geom. 83 (2005), no. 1-2, 22--31.

\bibitem[FF1]{ff1}S. Friedland, P. Freitas, 
p-metrics on GL(n,C)/Un and their Busemann compactifications. 
Linear Algebra Appl. 376 (2004), 1--18. 

\bibitem[FF2]{ff2}S. Friedland, P. Freitas, 
  Revisiting the Siegel upper half plane. I. Linear Algebra Appl. 376 (2004), 19--44.


\bibitem[Fu]{fu}W. Fulton, {\em Introduction to toric varieties.} 
Annals of Mathematics Studies, 131. The William H. Roever Lectures in Geometry. Princeton University Press, Princeton, NJ, 1993. xii+157 pp. 

\bibitem[Gr1]{gr}M. Gromov, Hyperbolic manifolds, groups and actions. {\em Riemann surfaces and related topics}: 
Proceedings of the 1978 Stony Brook Conference (State Univ. New York, Stony Brook, N.Y., 1978), pp. 183--213, Ann. of Math. Stud., 97, Princeton Univ. Press, Princeton, N.J., 1981.


\bibitem[Gr2]{gr2}M. Gromov, Filling Riemannian manifolds. 
J. Differential Geom. 18 (1983), no. 1, 1--147. 


\bibitem[Gru]{gru}P. Gruber,  {\em Convex and discrete geometry.}
 Grundlehren der Mathematischen Wissenschaften, 336. Springer, Berlin, 2007. xiv+578 pp.

\bibitem[GrL]{grl}P. Gruber, C. Lekkerkerker, 
{\em Geometry of numbers.} 
Second edition. North-Holland Mathematical Library, 37. North-Holland Publishing Co., Amsterdam, 1987. xvi+732 pp.


\bibitem[GJT]{gjt}Y. Guivarc'h, L. Ji, J.C. Taylor, {\em Compactifications of symmetric spaces.}
 Progress in Mathematics, 156. Birkh\"auser Boston, Inc., Boston, MA, 1998. xiv+284 pp.


\bibitem[Ha]{ha}T. Hattori, Busemann functions and positive eigenfunctions of Laplacian on noncompact symmetric spaces. 
J. Math. Kyoto Univ. 40 (2000), no. 3, 407--435. 

\bibitem[HSW]{hsw}T. Haettel, A. Schilling, A. Wienhard, {\em Horofunction Compactifications of Symmetric Spaces}. Preprint,  arXiv:1705.05026, 2017.

\bibitem[JM]{jm}L. Ji, R.  MacPherson, 
Geometry of compactifications of locally symmetric spaces. 
Ann. Inst. Fourier (Grenoble) 52 (2002), no. 2, 457--559. 



\bibitem[JS]{js}L. Ji, A. Schilling, {\em Polyhederal horofunction compactification as polyhedral ball}. Preprint, arXiv:1607.00564v2, 2016.

\bibitem[KL]{kl}A. Karlsson,  F. Ledrappier,  On laws of large numbers for random walks. Ann. Probab. 34 (2006), no. 5, 1693--1706.


\bibitem[KMN]{kmn}A. Karlsson, V. Metz, G. Noskov, 
Horoballs in simplices and Minkowski spaces. Int. J. Math. Math. Sci. 2006, Art. ID 23656, 20 pp.

\bibitem[KK]{kk}
A. Kasue,  H. Kumura, 
Spectral convergence of Riemannian manifolds.
Tohoku Math. J. (2) 46 (1994), no. 2, 147--179. 


\bibitem[KN]{kn}T. Klein, A.  Nicas,  The horofunction boundary of the Heisenberg group: the Carnot-Carath\'eodory metric. Conform. Geom. Dyn. 14 (2010), 269--295.

\bibitem[LN]{ln}B. Lemmens, R. Nussbaum,  Birkhoff's version of Hilbert's metric and its applications in analysis. 
{\em Handbook of Hilbert geometry}, 275--303, IRMA Lect. Math. Theor. Phys., 22, Eur. Math. Soc., Z\"urich, 2014. 

\bibitem[LW]{lw}B. Lemmens, C.  Walsh, 
Isometries of polyhedral Hilbert geometries. 
J. Topol. Anal. 3 (2011), no. 2, 213--241. 


\bibitem[LS]{ls}L. Liu, W. Su, The horofunction compactification of the Teichm\"uller metric. {\em Handbook of Teichm\"uller theory}. Vol. IV, 355--374, IRMA Lect.
 Math. Theor. Phys., 19, Eur. Math. Soc., Z\"urich, 2014.


\bibitem[Mi]{mi}E. Miller,  What is ... a toric variety? Notices Amer. Math. Soc. 55 (2008), no. 5, 586--587.

\bibitem[Od1]{od1}T. Oda, {\em  
Convex bodies and algebraic geometry. 
An introduction to the theory of toric varieties.}
 Translated from the Japanese. Ergebnisse der Mathematik und ihrer Grenzgebiete (3) [Results in Mathematics and Related Areas (3)], 15. Springer-Verlag, Berlin, 1988. viii+212 pp.


\bibitem[Od2]{od2}T. Oda, {\em 
Torus embeddings and applications.} 
Based on joint work with Katsuya Miyake. Tata Institute of Fundamental Research Lectures on Mathematics and Physics, 57. Tata Institute of Fundamental Research, Bombay; by Springer-Verlag, Berlin-New York, 1978. xi+175 pp. 

\bibitem[Ri1]{ri1}M. Rieffel,  Group $C^\star$-algebras as compact quantum metric spaces. Doc. Math. 7 (2002), 605--651. 

\bibitem[Ri2]{ri2}M. Rieffel, Metrics on state spaces. 
Doc. Math. 4 (1999), 559--600.

%\bibitem[Sch]{sch}A. Schlling, {\em Horofunction compactification of finite-dimensional
%normed spaces and of symmetric spaces}, Diploma Thesis at Heidelberg University, 2013.


\bibitem[So]{so}F. Sottile, Toric ideals, real toric varieties, and the moment map. {\em Topics in algebraic geometry and geometric modeling}, 225--240, Contemp. Math., 334, Amer. Math. Soc., Providence, RI, 2003.


\bibitem[Th]{th}A.C. Thompson, 
{\em Minkowski geometry.} Encyclopedia of Mathematics and its Applications, 63. Cambridge University Press, Cambridge, 1996. xvi+346 pp.


\bibitem[Ve]{ve}C. Vernicos, 
  On the Hilbert geometry of convex polytopes. {\em Handbook of Hilbert geometry,} 111--125, IRMA Lect. Math. Theor. Phys., 22, Eur. Math. Soc., Z\"urich, 2014.



\bibitem[Wa1]{wa1}C. Walsh,  The horofunction boundary of 
finite-dimensional normed spaces. Math. Proc. Cambridge Philos. Soc. 142 (2007), no. 3, 497--507.

\bibitem[Wa2]{wa2}C. Walsh, 
Busemann points of Artin groups of dihedral type.
Internat. J. Algebra Comput. 19 (2009), no. 7, 891--910. 


\bibitem[Wa3]{wa3}C. Walsh, 
The horofunction boundary of the Hilbert geometry.  
Adv. Geom. 8 (2008), no. 4, 503--529.



\bibitem[Wa4]{wa4}C. Walsh,  The horofunction boundary and isometry group of the Hilbert geometry. 
{\em Handbook of Hilbert geometry}, 127--146, IRMA Lect. Math. Theor. Phys., 22, Eur. Math. Soc., Z\"urich, 2014.

\bibitem[Wa5]{wa5}C. Walsh,  The horoboundary and isometry group of Thurston's Lipschitz metric. 
{\em Handbook of Teichm\"uller theory}. Vol. IV, 327--353, IRMA Lect. Math. Theor. Phys., 19, Eur. Math. Soc., Z\"urich, 2014.


\bibitem[WW1]{ww1}C. Webster,  A. Winchester, Busemann points of infinite graphs.  Trans. Amer. Math. Soc. 358
(2006), 4209â4224.


\bibitem[WW2]{ww2}C. Webster,  A. Winchester, 
Boundaries of hyperbolic metric spaces.
Pacific J. Math. 221 (2005), no. 1, 147--158. 


\end{thebibliography}
\end{document}